\documentclass[article,12pt]{amsart}

\usepackage{amsfonts}
\usepackage{amsmath}
\usepackage{mathrsfs}
\usepackage{graphicx}
\usepackage{color}
\usepackage{epsfig}
\usepackage{enumerate}

\numberwithin{equation}{section}
\theoremstyle{plain}
\newtheorem{thm}{Theorem}[section]
\newtheorem{rem}{Remark}[section]

\newtheorem{cor}{Corollary}[section]
\newtheorem{lem}{Lemma}[section]
\newtheorem{deff}{Definition}[section]

\newcommand{\dd}{\: \mathrm{d}}
\newcommand{\dE}{\mathbb{E}}
\newcommand{\dR}{\mathbb{R}}

\newcommand{\dP}{\mathbb{P}}

\newcommand{\cD}{\mathcal{D}}

\newcommand{\cN}{\mathcal{N}}

\newcommand{\rI}{\mathrm{I}}
\newcommand{\cF}{\mathcal{F}}
\newcommand{\cI}{\mathcal{I}}

\newcommand{\argmax}[1]{\underset{#1}{\operatorname{arg}\,\operatorname{max}}\;}
\newcommand{\wh}{\widehat}
\newcommand{\vp}{\varphi}

\newcommand{\ind}{\mbox{1}\kern-.25em \mbox{I}}

\def\build#1_#2^#3{\mathrel{\mathop{\kern 0pt#1}\limits_{#2}^{#3}}}

\def\videbox{\mathbin{\vbox{\hrule\hbox{\vrule height1.4ex \kern.6em\vrule height1.4ex}\hrule}}}
\def\demend{\hfill $\videbox$\\}

\email{Bernard.Bercu@u-bordeaux.fr}
\email{Adrien.Richou@u-bordeaux.fr}

\begin{document}
\title[Large deviations for the Ornstein-Uhlenbeck process without tears]
{Large deviations and concentration inequalities for the Ornstein-Uhlenbeck process without tears \vspace{1ex}}
\author{Bernard Bercu}
\address{Universit\'e de Bordeaux, Institut de Math\'ematiques de Bordeaux,
UMR 5251, 351 Cours de la Lib\'eration, 33405 Talence cedex, France.}
\author{Adrien Richou}
\address{
}
\thanks{}

\begin{abstract}
Our goal is to establish large deviations and concentration inequalities 
for the maximum likelihood estimator of the drift parameter of the Ornstein-Uhlenbeck 
process without tears. We propose a new strategy to establish large deviation results 
which allows us, via a suitable transformation, to circumvent the classical difficulty of non-steepness. 
Our approach holds in the stable case where the process is positive recurrent as well as
in the unstable and explosive cases where the process is respectively null recurrent and transient.
Notwithstanding of this trichotomy, we also provide new concentration inequalities 
for the maximum likelihood estimator.
\end{abstract}

\maketitle


\section{INTRODUCTION}


Consider the Ornstein-Uhlenbeck process observed over the time 
interval $[0,T]$
\begin{equation}
\label{OUPS}
dX_t=\theta X_t dt + dB_t
\end{equation}
where $(B_t)$ is a standard Brownian motion and the drift $\theta$ is an unknown real parameter. For the sake of
simplicity, we assume that the initial state $X_0 = 0$. The process is said to be stable if $\theta < 0$, 
unstable if $\theta = 0$, and explosive if $\theta > 0$. The maximum likelihood estimator of $\theta$ is given by
\begin{equation} 
\label{DEFTHETAHAT}
\widehat{\theta}_T = \frac{\int_0^T X_t dX_t}{\int_0^T X_t^2 dt} = \frac{X_T^2 - T}{2\int_0^T X_t^2 dt}.
\end{equation}
It is well-known that in the stable, unstable, and explosive cases
\begin{equation*} 
\lim_{T\rightarrow \infty}\widehat{\theta}_T = \theta \hspace{1cm}\text{a.s.}
\end{equation*}
The purpose of this paper is to establish large deviation principles (LDP) and concentration inequalities (CI)
for $(\wh{\theta}_T)$ via fairly easy to handle arguments.
\\ \vspace{-1ex}\par
In the stable case, Florens-Landais and Pham \cite{FP99} proved an LDP for the score function
and they were able to deduce, by contraction, the LDP for $(\wh{\theta}_T)$. However, one can realize in Lemma 4.3 of \cite{FP99}
that the normalized cumulant generating function of the score function is quite complicated to
compute. Moreover, its LDP relies on a sophisticated time varying change of probability.
\\ \vspace{-1ex}\par
In the unstable and explosive cases \cite{BCS12}, the strategy for proving an LDP for $(\wh{\theta}_T)$ is also
far from being obvious. As a matter of fact, on can observe in Lemma 2.1 of \cite{BCS12} that
the normalized cumulant generating function is also very complicated to evaluate. Moreover, as
the limiting cumulant generating function is not steep, it is also necessary to make use of a sophisticated 
time varying change of probability.
\\ \vspace{-1ex}\par
Our approach is totally different. It will allows us, via a suitable transformation,
to circumvent the classical difficulty of non-steepness. The starting point is to establish, thanks
to G\"artner-Ellis's theorem \cite{DZ98}, an LDP for the
couple
\begin{equation}
\label{COUPLE}
V_T=\Bigl(\frac{X_T}{\sqrt{T}},\frac{S_T}{T}\Bigr)
\end{equation}
where the energy $S_T$ is given by
$$
S_T= \int_0^T X_t^2dt.
$$ 
Then, we will obtain the LDP for $(\wh{\theta}_T)$ by a direct use of the contraction principle.
We refer the reader to the recent paper \cite{BR15} where this strategy was successfully implemented
for the Ornstein-Uhlenbeck process with shift.
\\ \vspace{-1ex}\par
Furthermore, to the best of our knowledge, very few results are available on CI except in the stable case
for the energy  \cite{GGW14}, \cite{L01}. In addition, one can observe that Theorem 1.1 of Gao and Jiang \cite{GJ09}
can be significantly improved, even in the special case where the shift parameter is zero.
Our second goal is to fill the gap by proving CI for $(\wh{\theta}_T)$
in the stable, unstable, and explosive cases.
\\ \vspace{-1ex}\par
The paper is organized as follows. In Section \ref{S-LDP}, we establish an LDP for the couple given 
by \eqref{COUPLE} and we deduce by contraction the LDP for $(\wh{\theta}_T)$ in the stable, unstable, 
and explosive cases. Section \ref{S-CI} is devoted to CI for $(\wh{\theta}_T)$.
Standard tools for proving LDP such as the G\"artner-Ellis theorem and 
the contraction principle are recalled in Appendix A, while
all technical proofs of Sections \ref{S-LDP} and \ref{S-CI} 
are postponed to Appendices B and C.


\section{Large deviations.}
\setcounter{equation}{0}
\label{S-LDP}


The usual notions of full and weak LDP are as follows.

\begin{deff}
A sequence of random vectors $(V_T)$ of $\dR^d$ satisfies an LDP with speed $T$ and
rate function $I$ if $I$ is a lower semicontinuous function from $\dR^d$ to $[0,+\infty]$ such that,
\begin{enumerate}[(i)] 
\item $\!$Upper bound: For any closed set $F  \subset \dR^d$,
\begin{equation}
\label{LDPUB}
 \limsup_{T\rightarrow\infty}\frac{1}{T}\log \dP\bigl(V_T\in F\bigr)\leq
-\inf_{x\in F}I(x).
\end{equation} 
\item $\!$Lower bound: For any open set $G \subset \dR^d$,
\begin{equation}
\label{LDPLB}
-\inf_{x\in G}I(x)\leq \liminf_{n\rightarrow\infty}\frac{1}{T}\log \dP\bigl(V_T\in G\bigr).
\end{equation}
\end{enumerate}
Moreover, $I$ is said to be a good rate function if its level sets are compact.
\end{deff}

\begin{deff}
A sequence of random vectors $(V_T)$ of $\dR^d$ satisfies a weak LDP with speed $T$ and
rate function $I$ if $I$ is a lower semicontinuous function from $\dR^d$ to $[0,+\infty]$ such that
the upper bound \eqref{LDPUB} holds for any compact set, while the lower bound \eqref{LDPLB}
is true for any open set.
\end{deff}

It is well-known that if $(V_T)$ is exponentially tight and satisfies
a weak LDP, then $I$ is a good rate function and the full LDP holds for $(V_T)$,
see Lemma 1.2.18 of \cite{DZ98}. 

\subsection{The stable case}

First of all, we focus our attention on the easy to handle stable case where the parameter $\theta$ is negative
in \eqref{OUPS}.

\begin{thm}
\label{T-LDP-SC}
The couple $(V_T)$, given by \eqref{COUPLE}, satisfies an LDP with speed $T$ and good rate function
$\cI_{\theta}$ given by
\begin{equation}
\label{RATE-SC}
 \cI_{\theta}(x,y)= \left\{ \begin{array}{ccl}
             {\displaystyle \frac{\theta(1 - x^2 + \theta y)}{2}} + {\displaystyle \frac{(1+x^2)^2}{8y}}
              & \textrm{if} & y>0,\vspace{1ex}\\
	     + \infty  & \textrm{if} &   y\leq 0.
            \end{array} \right.
\end{equation}
\end{thm}

\noindent
We clearly deduce from \eqref{DEFTHETAHAT} that
\begin{equation}
\label{CSTHETAHAT}
\wh{\theta}_T=f(V_T)
\end{equation}
where $f$ is the continuous function defined, for all $x\in \dR$ and for any positive $y$, by
$$
f(x,y)= \frac{x^2 - 1}{2y}.
$$
Hence, an elementary application of the contraction principle 
given in Appendix A, leads to the following corollary, which was previously established in \cite{FP99}
via a much more complicated strategy.

\begin{cor}
\label{C-LDP-ST}
The sequence $(\wh{\theta}_T)$ satisfies an LDP with good rate function
\begin{equation}
\label{RATE-UC}
 I_{\theta}(z)= \left\{ \begin{array}{ccc}
             -{\displaystyle \frac{(z-\theta)^2}{4z}} & \textrm{if} & {\displaystyle z \leq \frac{\theta}{3}},\vspace{1ex}\\
             2z-\theta & \textrm{if} &  {\displaystyle z\geq \frac{\theta}{3}}.
            \end{array} \right.
\end{equation}
\end{cor}

\begin{proof}
The proofs are given is Appendix B.
\end{proof}

\subsection{The unstable case}

Hereafter, we carry out our strategy on the unstable case where the parameter $\theta=0$ 
in \eqref{OUPS}.

\begin{thm}
\label{T-LDP-UC}
The couple $(V_T)$, given by \eqref{COUPLE}, satisfies a weak LDP with speed $T$ and good rate function
$\cI_{\theta}$ given by
\begin{equation}
\label{RATE-US}
 \cI_{0}(x,y)= \left\{ \begin{array}{ccl}
             {\displaystyle \frac{(1+x^2)^2}{8y}}
              & \textrm{if} & y>0,\vspace{1ex}\\
	     + \infty  & \textrm{if} &   y\leq 0.
            \end{array} \right.
\end{equation}
\end{thm}

\noindent
Despite the lack of exponential tightness, it is possible to establish
the following corollary, which was previously proved in \cite{BCS12}
via a much more complex procedure.

\begin{cor}
\label{C-LDP-UT}
The sequence $(\wh{\theta}_T)$ satisfies an LDP with good rate function
\begin{equation}
\label{RATEC}
 I_{0}(z)= \left\{ \begin{array}{ccc}
            {\displaystyle \hspace{-1ex} -\frac{z }{4} } & \text{ if } & z \leq 0, \vspace{1ex}\\
     2z   &\text{ if } & z \geq 0.
            \end{array} \right.
\end{equation}
\end{cor}

\begin{proof}
The proofs are given is Appendix B.
\end{proof}

\subsection{The explosive case}

Finally, we deal with the more complicated explosive case where the parameter $\theta$ is
positive in \eqref{OUPS}.
\begin{thm}
\label{T-LDP-EC}
The couple $(V_T)$, given by \eqref{COUPLE}, satisfies a weak LDP with speed $T$. More precisely,
\begin{enumerate}[(i)]
\item $\!$ Upper bound: For any compact set $F  \subset \dR^2$,
\begin{equation}
\label{LDPUBE}
 \limsup_{T\rightarrow\infty}\frac{1}{T}\log \dP\bigl(V_T\in F\bigr)\leq
-\inf_{(x,y)\in F}\cI_{\theta}(x,y).
\end{equation} 
\item $\!$ Lower bound: For any open set $G \subset \dR^2$,
\begin{equation}
\label{LDPLBE}
-\inf_{(x,y)\in G \cap \cF}\cI_{\theta}(x,y)\leq \liminf_{n\rightarrow\infty}\frac{1}{T}\log \dP\bigl(V_T\in G\bigr),
\end{equation}
\end{enumerate}
where $\cI_{\theta}$ is the good rate function given by
\begin{equation}
\label{RATE-EC}
 \cI_{\theta}(x,y)= \left\{ \begin{array}{ccl}
             {\displaystyle \frac{\theta(1 - x^2 + \theta y)}{2} +\frac{(1+x^2)^2}{8y}}
              & \textrm{if} & \displaystyle{ 0<y<\frac{1}{2\theta}(1+x^2),}\\
              \theta & \textrm{if} & y \geq {\displaystyle \frac{1}{2\theta}(1+x^2),}\vspace{2ex} \\
	     + \infty  & \textrm{if} &   y\leq 0,
            \end{array} \right.
\end{equation}
and $\cF$ is the set of exposed points of $\cI_{\theta}$ defined by
\begin{equation}
 \label{EQ-cF-EC}
 \cF = \Bigl\{ (x,y) \in \mathbb{R}^2 \textrm{ such that } 0 <y < \frac{1}{2\theta}(1+x^2) \Bigr\}.
\end{equation}
\end{thm}

\begin{rem}
\label{REM-TthetaC0}
 Let us remark that $\cI_{\theta}$ is a continuous function on $\mathbb{R} \times \mathbb{R}^{*}_+$ and a 
 constant function on $(\mathbb{R} \times \mathbb{R}^{*}_+) \setminus \mathcal{F}$. Consequently, we 
 are able to precise \eqref{LDPLBE}: For any open set $G \subset \dR^2$ such that $G \cap \cF \neq \emptyset$,
 $$\inf_{(x,y)\in G \cap \cF}\cI_{\theta}(x,y) = \inf_{(x,y)\in G }\cI_{\theta}(x,y).$$
\end{rem}

\noindent
Despite the weak large deviation result of Theorem \ref{T-LDP-EC}, it is possible to establish
the following corollary, which was previously proved in \cite{BCS12}
via a much more complex procedure.

\begin{cor}
\label{C-LDP-EC}
The sequence $(\wh{\theta}_T)$ satisfies an LDP with good rate function
\begin{equation}
\label{RATE-ET}
 I_{\theta}(z)= \left\{ \begin{array}{ccc}
            {\displaystyle \hspace{-1ex} -\frac{(z-\theta)^2 }{4z} } & \text{ if } & z \leq -\theta, \vspace{1ex}\\
     \theta   &\text{ if } & |z| < \theta, \vspace{1ex}\\
      0  &\text{ if } & z = \theta, \vspace{1ex}\\
      2z-\theta  &\text{ if } & z > \theta.
            \end{array} \right.
\end{equation}
\end{cor}

\begin{proof}
The proofs are given is Appendix B.
\end{proof}


\section{Concentration inequalities.}
\setcounter{equation}{0}
\label{S-CI}


The concentration inequalities for the sequence $(\wh{\theta}_T)$
are gathered together as follows. We refer the reader to \cite{BDR15}
for a recent book on concentration inequalities for sums and martingales.

\begin{thm}
\label{T-CI}
We have for all positive real numbers $T$ and for any positive $x$,
\begin{equation}
\label{CIGEN}
\dP(| \wh{\theta}_T - \theta | \geq x) \leq 2 
\exp\Bigl( -\frac{x^2}{2} h_T(y_x)\Bigr)
\end{equation}
where
\begin{equation}
\label{RATECI}
 h_T(y)=
 \left\{ \begin{array}{ccc}
 {\displaystyle  \frac{-\theta T y + \log(y+2) -\log(2(y+1))}{x^2 +\theta^2y(y+2)}} 
 & \text{ if } & \theta< 0, \vspace{1ex}\\
 {\displaystyle  \frac{ T y - \log 2}{x^2 +y^2}} 
 & \text{ if } & \theta= 0, \vspace{1ex}\\
 {\displaystyle  \frac{\theta T(y+2) + \log y -\log(2(y+1))}{x^2 +\theta^2y(y+2)}} 
 & \text{ if } & \theta> 0, \vspace{1ex}\\
            \end{array} \right.
\end{equation}
and
\begin{equation*}
y_x= \argmax{y>0} h_T(y).
\end{equation*}
\end{thm}

\begin{cor}
\label{C-CI-SC}
In the stable case $\theta < 0$, we have for all positive real numbers $T$ and for any positive $x$,
\begin{equation}
\label{CI-SC}
\dP(| \wh{\theta}_T -\theta | \geq x) \leq 2 
\exp\Bigl( -\frac{T^2x^2}{4(\log 2 -\theta T+ \sqrt{T^2x^2 -2\theta T \log 2+(\log 2)^2})}\Bigr).
\end{equation}
In particular, as soon as $0< x \leq - \theta$,
\begin{equation*}
\dP(| \wh{\theta}_T - \theta | \geq x) \leq 2 
\exp\Bigl( -\frac{T^2x^2}{8(-\theta T +\log 2)}\Bigr),
\end{equation*}
while, for any $x>-\theta$,
\begin{equation*}
\dP(| \wh{\theta}_T - \theta | \geq x) \leq 2 
\exp\Bigl( -\frac{T^2x^2}{4(T(x-\theta) + 2\log 2)}\Bigr).
\end{equation*}
\end{cor}

\begin{cor}
\label{C-CI-UC}
In the unstable case $\theta=0$, we have for all positive real numbers 
$T$ and for any positive $x$,
\begin{equation}
\label{CI-UC}
\dP(| \wh{\theta}_T | \geq x) \leq 2 
\exp\Bigl( -\frac{T^2x^2}{4(\log 2 + \sqrt{T^2x^2 +(\log 2)^2})}\Bigr).
\end{equation}
In particular,
\begin{equation*}
\dP(| \wh{\theta}_T | \geq x) \leq 2 
\exp\Bigl( -\frac{T^2x^2}{4(Tx +2\log 2)}\Bigr).
\end{equation*}
\end{cor}

\begin{cor}
\label{C-CI-EC}
In the explosive case $\theta>0$, we have for all positive real numbers $T$ and for 
any positive $x$,
\begin{equation}
\label{CI-EC-xpetit}
\dP(| \wh{\theta}_T - \theta | \geq x) \leq 2 
\exp\Bigl( -\frac{T^2x^2\bigl(2\theta T +\log (\log 2)-\log(\theta T+\log 2)\bigr)}{2(T^2x^2+
2\theta T \log 2+(\log 2)^2)}\Bigr).
\end{equation}
\end{cor}

\section*{Appendix A \\ G\"artner-Ellis theorem and 
the contraction principle.}
\renewcommand{\thesection}{\Alph{section}}
\renewcommand{\theequation}{\thesection.\arabic{equation}}
\setcounter{section}{1}
\setcounter{equation}{0}

The most powerful tool for proving LDP is probably the G\"artner-Ellis theorem.
Let $(V_T)$ be sequence of random vectors of $\dR^d$. Denote by $L_T$ the
normalized cumulant generating function of $V_T$,
$$L_T(a) = \frac{1}{T} \log \dE \left[\exp \left( T \langle a, V_T \rangle \right) \right].$$
The existence of the limiting cumulant generating function
$$
L(a)=\lim_{T \rightarrow \infty} L_T(a)
$$
indicates whether or not $(V_T)$ satisfies an LDP. Denote by $\cD_L$ the effective domain of $L$,
$$
\cD_L = \bigl\{ a \in \dR^d \text{ such that } L(a)< \infty \bigr\}.
$$
Let $I$ be the Fenchel-Legendre transform of $L$,
$$
I(x)=\sup_{a \in \dR^d} \bigl\{ \langle a, x \rangle -L(a) \bigr\}.
$$

\begin{thm}[G\"artner-Ellis]
\label{T-GE}
Assume that the function $L$ exists as an extended real number. Then,
\begin{enumerate}[(i)]
\item $\!$ Upper bound: For any compact set $F  \subset \dR^d$,
\begin{equation}
\label{LDPUBGE}
 \limsup_{T\rightarrow\infty}\frac{1}{T}\log \dP\bigl(V_T\in F\bigr)\leq
-\inf_{x\in F}I(x).
\end{equation} 
\item $\!$Lower bound: For any open set $G \subset \dR^d$,
\begin{equation}
\label{LDPLBGE}
-\inf_{x\in G \cap \cF}I(x)\leq \liminf_{n\rightarrow\infty}\frac{1}{T}\log \dP\bigl(V_T\in G\bigr),
\end{equation}
where $\cF$ is the set of exposed points of $I$ whose exposing hyperplane belongs
to the interior of $\cD_L$. 
\item
If $L$ is an essentially smooth, lower semicontinuous function, 
then the sequence $(V_T)$ satisfies a weak LDP with rate function $I$. If, moreover,
the origin belongs to the interior of $\cD_L$, $(V_T)$ satisfies an LDP with good
rate function $I$.
\end{enumerate}
\end{thm}

\noindent
We refer the reader to the excellent book \cite{DZ98} for more insight 
on the theory of large deviations. In particular, the G\"artner-Ellis is given in Theorem 2.3.6
of \cite{DZ98}. Another useful tool is the contraction principle which ensures that an LDP
remains valid by continuous mapping, see Theorem 4.2.1 of \cite{DZ98}.

\begin{thm}[Contraction principle]
\label{T-CP}
Assume that a sequence of random vectors $(V_T)$ of $\dR^d$ satisfies 
an LDP with good rate function $I$, and that
$A_T=f(V_T)$
where $f$ is a continuous function from $\dR^d$ to $\dR^\delta$. Then, $(A_T)$ also satisfies an LDP with 
good rate function $J$ given, for all $y \in \dR^{\delta}$, by
\begin{equation}
\label{CONTRACT}
J(y)=\inf \bigl\{ I(x) \text{ with } x \in \dR^d \text{ such that } f(x)=y \bigr\},
\end{equation}
where the infimum over the empty set is taken to be infinite.
\end{thm}


\section*{Appendix B \\ Proofs of LDP results.}
\renewcommand{\thesection}{\Alph{section}}
\renewcommand{\theequation}{\thesection.\arabic{equation}}
\setcounter{section}{2}
\setcounter{equation}{0}
\label{PRLDP}

Let $L_T$ be the
normalized cumulant generating function of the couple
$$V_T=\Bigl(\frac{X_T}{\sqrt{T}},\frac{S_T}{T}\Bigr)$$
defined, for all $(a,b) \in \dR^2$, by
$$L_T(a,b) = \frac{1}{T} \log \dE \Bigl[\exp \Bigl( a\sqrt{T} X_T + bS_T \Bigr) \Bigr].$$
The proofs of all the LDP results rely on an accurate evaluation of $L_T(a,b)$ as well as on
the existence of the limiting cumulant generating function $L(a,b)$. This is the subject of the following keystone lemma.  
\begin{lem}
\label{L-KLEMMA}
In the stable and unstable cases $\theta \leq 0$, the effective domain of $L$ is given by
\begin{equation}
\label{DOMLSU}
\cD_L = \Bigl\{ (a,b) \in \dR^2 \text{ such that } b < \frac{\theta^2}{2} \Bigr\},
\end{equation}
while, in the explosive case $\theta >0$, the effective domain of $L$ becomes
\begin{equation}
\label{DOMLE}
\cD_L = \Bigl\{ (a,b) \in \dR^2 \text{ such that } b < 0 \Bigr\}.
\end{equation}
Moreover, for any $(a,b) \in \cD_L$, we have whatever the value of $\theta$,
\begin{equation}
\label{DEFLS}
L(a,b) = -\frac{1}{2}\left(\theta  +  \sqrt{\theta^2-2b} \right)
+\frac{a^2}{2(\sqrt{\theta^2-2b} - \theta)}.
\end{equation}
\end{lem}

\begin{rem}
One can observe that, as soon as $\theta \geq 0$, the origin does not belong to the interior of $\mathcal{D}_L$.
\end{rem}

\begin{proof}
We start with the stable and unstable cases. Using the same lines as in Appendix A of \cite{BR15}, we obtain 
from Girsanov's formula associated with \eqref{OUPS} that
\begin{eqnarray*}
L_T(a,b) &\!=\!& \frac{1}{T} \log \dE_{\vp} \!\left[ \exp\Bigl((\theta - \vp) \!\int_0^T \!\! X_t dX_t -\frac{1}{2}( \theta^2 - \vp^2) S_T 
+a\sqrt{T} X_T + bS_T\Bigr) \right], \\
&\!=\!& \frac{1}{T} \log \dE_{\vp} \!\left[ \exp\Bigl(\frac{(\theta - \vp)}{2} (X_T^2-T) +a\sqrt{T} X_T
+\frac{1}{2}( 2 b  - \theta^2 + \vp^2) S_T\Bigr) \right]
\end{eqnarray*}
where $\dE_{\vp}$ stands for the expectation after the usual change of probability, 
$$
\frac{\dd \dP_{\vp}}{\dd \dP_{\theta}}= 
\exp \left( ( \vp- \theta) \int_0^T X_t dX_t - \frac{1}{2} (\vp^2 - \theta^2) \int_0^T X_t^2\,dt 
\right).
$$
Consequently, if $\theta^2-2b>0$ and $\vp=\sqrt{\theta^2-2b}$, 
$L_T(a,b)$ reduces to
\begin{equation}
\label{DECMGFLS1}
L_T(a,b) = \frac{\vp - \theta }{2}+ \frac{1}{T} \log \dE_{\vp} 
\left[ \exp\Bigl( \Bigl(\frac{\theta - \vp}{2}\Bigr) X_T^2 +a\sqrt{T} X_T \Bigr) \right].
\end{equation}
Under the new probability $\dP_{\vp}$, $X_T$ has an $\cN(0,\sigma_T^2)$ distribution where
\begin{equation}
\label{VARS}
\sigma_T^2= \frac{1}{2 \vp} \Bigl( e^{2 \vp T} - 1 \Bigr).
\end{equation}
Hence, it follows from straightforward Gaussian calculations that
\begin{equation}
\label{DECMGFLS2}
L_T(a,b) = \frac{\vp-\theta}{2} + \frac{a^2 \sigma_T^2}{2\gamma_T}  - \frac{1}{2T} \log \gamma_T
\end{equation}
where $\gamma_T=1+(\vp - \theta) \sigma_T^2$. However, we clearly obtain from \eqref{VARS} that
$$
\lim_{T \rightarrow \infty} \frac{1}{T} \log\sigma_T^2=2\vp, \hspace{0.5cm}
\lim_{T \rightarrow \infty} \frac{\gamma_T}{\sigma_T^2} =\vp - \theta, \hspace{0.5cm}
\lim_{T \rightarrow \infty} \frac{1}{T} \log\gamma_T=2\vp.
$$
Hence, we deduce from \eqref{DECMGFLS2} that
\begin{equation}
\label{LIMLTS}
\lim_{T \rightarrow \infty} L_T(a,b) =  -\frac{1}{2}\left(\theta  + \vp \right)
+\frac{a^2}{2(\vp - \theta)},
\end{equation}
which is exactly the limiting cumulant generating function $L(a,b)$ given by \eqref{DEFLS}. 
In the explosive case $\theta>0$, calculations are quite the same with the only significant modification
that $\varphi = -\sqrt{\theta^2-2b}$ instead of $\sqrt{\theta^2-2b}$. Then, \eqref{DECMGFLS2} holds true
with the new parameter $\varphi$ and
$$
\lim_{T \rightarrow \infty} \frac{1}{T} \log\gamma_T=0, \hspace{0.5cm}
\lim_{T \rightarrow \infty} \frac{\gamma_T}{\sigma_T^2} =-(\vp + \theta).
$$
Consequently, \eqref{DEFLS} follows from \eqref{DECMGFLS2},
completing the proof of Lemma \ref{L-KLEMMA}.
\end{proof}

\noindent
We shall also make use of normalized cumulant generating function $\Lambda_T$ of the couple
\begin{equation*}
W_T=\Bigl(\frac{X_T^2}{T},\frac{S_T}{T}\Bigr)
\end{equation*}
defined, for all $(a,b) \in \dR^2$, by
$$\Lambda_T(a,b)=\frac{1}{T}\log \dE \Bigl [ \exp \Bigl(aX_T^2+bS_T\Bigr) \Bigr ].$$
The proofs of LDP results in the unstable and explosive cases require the following lemma on the 
effective domain of the limiting cumulant generating function $\Lambda(a,b)$ of 
$\Lambda_T(a,b)$.

\begin{lem}
 \label{L-LCGFXcarre}
If $\theta \geq 0$, the effective domain of $\Lambda$ is given by 
 \begin{equation*}
  \mathcal{D}_\Lambda=\Bigl\{ (a,b) \in \mathbb{R}^2 \textrm{ such that } \theta^2-2b>0 \textrm{ and } 2a+\theta<\sqrt{\theta^2-2b} \Bigr\}. 
 \end{equation*}
\end{lem}
\begin{proof}
The proof is the same as that of Lemma \ref{L-KLEMMA}
\end{proof}

\vspace{-4ex}
\subsection*{}
\begin{center}
{\rm {\small B.1. THE STABLE CASE.}}
\end{center}


\subsection*{Proof of Theorem \ref{T-LDP-SC}.}
The origin belongs to the interior of the domain $\cD_L$ given by \eqref{DOMLSU}. Moreover,
the function $L$, defined in \eqref{DEFLS}, is differentiable throughout $\cD_L$ and $L$ is steep, 
which means that $L$ is essentially smooth. Hence, one can immediately deduce from
the G\"artner-Ellis theorem that the couple $(V_T)$ satisfies an LDP with speed $T$ and good rate function 
$$
\cI_{\theta}(x,y)=\sup_{(a,b) \in \cD_L} \bigl\{ ax+by -L(a,b) \bigr\}.
$$
It is easy to compute $\cI_{\theta}$. After some straightforward calculations, we obtain the expression given by \eqref{RATE-SC},
which achieves the proof of Theorem \ref{T-LDP-SC}.
\demend
\vspace{-1ex}

\subsection*{Proof of Corollary \ref{C-LDP-ST}.}
Corollary \ref{C-LDP-ST} follows from Theorem \ref{T-LDP-SC} together with an elementary application
of the contraction principle. We already saw in Section \ref{S-LDP} that
$\wh{\theta}_T=f(V_T)$ where $f$ is the continuous function defined, for all $x\in \dR$ 
and for any positive $y$, by
$$
f(x,y)= \frac{x^2 - 1}{2y}.
$$
Consequently, one can immediately deduce from the contraction principle that the sequence $(\wh{\theta}_T)$
satisfies an LDP with good rate function $I_{\theta}$ given, for all $z \in \dR$, by
\begin{equation}
\label{CS1}
I_{\theta}(z) = \inf \Bigl\{ \cI_{\theta}(x,y) \text{ with } x \in \dR, y>0 \text{ such that } f(x,y)=z \Bigr\}.
\end{equation}
Hereafter, it only remains to properly evaluate $I_{\theta}$. As soon as $1+2yz \geq 0$,
$$
f(x,y)=z \iff x^2=1+2yz.
$$
Hence, \eqref{RATE-SC} together with \eqref{CS1} lead to
\begin{equation}
\label{CS2}
I_{\theta}(z) = \inf \bigl\{ h(y) \text{ with } 1+2yz \geq 0, y>0 \bigr\}
\end{equation}
where $h$ is the function defined, for any positive $y$, by
\begin{equation}
\label{CS3}
h(y) = \frac{\theta y(\theta -2z)}{2} +  \frac{(1+yz)^2}{2y}.
\end{equation}
We clearly have from \eqref{CS3} that $h$ is a convex function as
\begin{equation}
\label{CS4}
h^{\prime}(y)= \frac{1}{2}\Bigl( (z-\theta)^2 - \frac{1}{y^2}\Bigr)
\hspace{1cm}\text{and}\hspace{1cm}
h^{\prime \prime}(y)= \frac{1}{y^3}.
\end{equation}
The evaluation of the rate function $I_{\theta}$ depends on the location of
its argument. 
On the one hand, as soon as $z\leq \theta/3$, 
the border condition $1+2yz \geq 0$ plays a prominent role as
$$
I_{\theta}(z)=h\Bigl( -\frac{1}{2z}  \Bigr)= -\frac{(z-\theta)^2}{4z}.
$$
On the other hand, as soon as $z\geq \theta/3$, the border condition $1+2yz \geq 0$
does not have to be taken into account as
$$
I_{\theta}(z)=h\Bigl( \frac{1}{z-\theta}  \Bigr)= 2z - \theta,
$$
which completes the proof of Corollary \ref{C-LDP-ST}.
\demend

\vspace{-4ex}
\subsection*{}
\begin{center}
{\rm {\small B.2. THE UNSTABLE CASE.}}
\end{center}

\subsection*{Proof of Theorem \ref{T-LDP-UC}.}
The proof of Theorem \ref{T-LDP-UC} can be handled exactly as that of Theorem \ref{T-LDP-SC}
by taking the value $\theta = 0$. The function $L$, given by \eqref{DEFLS}, is essentially smooth.
However, in contrast with the stable case, the origin does no longer belong to 
the interior of $\cD_L$. It means that the sequence $(V_n)$ is not 
exponentially tight. This is the reason why we obtain a weak LDP for $(V_n)$ instead
of a full LDP, via the weak version of the G\"artner-Ellis Theorem \ref{T-GE}.
\demend

\subsection*{Proof of Corollary \ref{C-LDP-UT}.}
Since Theorem \ref{T-LDP-UC} provides us a weak LDP for the sequence $(V_n)$, we cannot deduce
Corollary \ref{C-LDP-UT} from a direct application of the contraction principle.
Instead of that, we shall prove the LDP for $(\wh{\theta}_T)$ 
by considering the rare events $\{\wh{\theta}_T \leq c\}$ and $\{\wh{\theta}_T \geq c\}$,
for $c$ negative and $c$ positive, respectively.
First of all, we have for any negative $c$,
\begin{equation*}
 \dP(\wh{\theta}_T \leq c) = \dP(f(V_T) \leq c)= \dP(V_T \in \Delta_c)
\end{equation*}
where the set $\Delta_c$ is given, for $a_c(x)=(x^2 -1)/2c$, by 
$$\Delta_c = \Bigl\{ (x,y) \in \dR^2 \text{ such that } |x| \leq 1 \text{ and } y \in [0,a_c(x)]\Bigr\}.$$
Since $\Delta_c$ is a compact set of $\dR^2$, Theorem \ref{T-LDP-UC} implies that
$$\lim_{T \rightarrow + \infty} \frac{1}{T} \log \mathbb{P}(\wh{\theta}_T \leq c) = -\inf_{(x,y) \in \Delta_c} \cI_0(x,y).$$
However, the rate function $\cI_0$ has no critical points and $\cI_0(x,0)=+\infty$. Hence,
$$\lim_{T \rightarrow + \infty} \frac{1}{T} \log \dP(\wh{\theta}_T \leq c) 
= -\inf_{|x| <1} \cI_0(x, a_c(x)) 
= -\cI_0\left(0,-\frac{1}{2c}\right) 
= \frac{c}{4}=-I_0(c).$$
We now consider the more tedious case where $c$ is positive. We have for any $\alpha >0$, 
\begin{equation}
\label{DECOUC}
 \dP(\wh{\theta}_T \geq c) = \dP\Bigl(\wh{\theta}_T \geq c,\frac{|X_T|}{\sqrt{T}} \leq \alpha \Bigr) +  
 \dP\Bigl(\wh{\theta}_T \geq c,\frac{|X_T|}{\sqrt{T}} > \alpha \Bigr).
\end{equation}
One can remark that 
$$\dP\Bigl(\wh{\theta}_T \geq c,\frac{|X_T|}{\sqrt{T}} \leq \alpha \Bigr) = \dP (V_T \in \Delta_{c,\alpha})$$
where $\Delta_{c,\alpha}$ is the compact set of $\mathbb{R}^2$ defined by
$$\Delta_{c,\alpha}= \Bigl\{ (x,y) \in \dR^2 \text{ such that } 1 \leq |x| \leq \alpha \text{ and } y \in [0,a_c(x)]\Bigr\}.$$
Therefore, we deduce from Theorem \ref{T-LDP-UC} that
$$\lim_{T \rightarrow + \infty} \frac{1}{T} \log \dP(V_T \in \Delta_{c,\alpha}) = -\inf_{(x,y) \in \Delta_{c,\alpha}} \cI_0(x,y).$$
After some straightforward calculations, we obtain that, as soon as $\alpha\ge\sqrt{3}$,
\begin{equation}
\label{PRUT1}
\lim_{T \rightarrow + \infty} \frac{1}{T} \log \dP\Bigl(\wh{\theta}_T \geq c,\frac{|X_T|}{\sqrt{T}} \leq \alpha \Bigr) = -\cI_0\Bigl(\sqrt{3},\frac{1}{c}\Bigr) = -2c = -I_0(c).
\end{equation}
It only remains to prove that the right-hand side of \eqref{DECOUC} is negligeable. It follows from Markov's inequality
that for any positive $\lambda$ and $\mu$,
\begin{eqnarray}
\dP\Bigl(\wh{\theta}_T \geq c,\frac{|X_T|}{\sqrt{T}} > \alpha \Bigr) &=& \dP\Bigl( X_T^2 -2cS_T \geq T,
 X_T^2 > \alpha^2 T\Bigr), \notag \\
 & \leq & \exp\Bigl(-T(\lambda + \mu \alpha^2)\Bigr) \dE\Bigl[ \exp\Bigl( (\lambda+\mu)X_T^2 -2\lambda cS_T\Bigr) \Bigl], \notag \\
 & \leq & \exp\Bigl(-T\Bigl((\lambda + \mu \alpha^2)-\Lambda_T(\lambda+\mu, -2 \lambda c) \Bigr) \Bigr). 
\label{PRUT2}
\end{eqnarray}
By choosing $\lambda=\mu=c/5$, it is not hard to see that the couple $(2c/5,-2c^2/5)$ belongs to the effective domain
$\cD_\Lambda$ given in Lemma \ref{L-LCGFXcarre}. Hence, as $\Lambda_T$ converges simply to $\Lambda$ on $\cD_\Lambda$,
we infer from \eqref{PRUT2} that for $T$ large enough,
\begin{equation*}
\dP\Bigl(\wh{\theta}_T \geq c,\frac{|X_T|}{\sqrt{T}} > \alpha \Bigr) \leq \exp\Bigl(-T\Bigl(\frac{c}{5}(1 +\alpha^2)-
2\Lambda\Bigl(\frac{2c}{5}, -\frac{2c^2}{5}\Bigr) \Bigr) \Bigr). 
\end{equation*}
which implies that for $\alpha$ and $T$ large enough, 
\begin{equation}
\label{PRUT3}
 \dP\Bigl(\wh{\theta}_T \geq c,\frac{|X_T|}{\sqrt{T}} > \alpha \Bigr) \leq \exp(-3cT).
\end{equation}
Therefore, it follows from the conjunction of \eqref{DECOUC}, \eqref{PRUT1} and \eqref{PRUT3} that
for any positive $c$,
$$\lim_{T \rightarrow + \infty} \frac{1}{T}\log \dP(\wh{\theta}_T \geq c) = -2c = -I_0(c).$$
Finally, in the unstable case, $X_T$ has an $\cN(0,T)$ distribution. Hence,
the case $c=0$ is straightforward as
$$\lim_{T \rightarrow + \infty} \frac{1}{T}\log \dP(\wh{\theta}_T \geq 0) = 
\lim_{T \rightarrow + \infty} \frac{1}{T}\log \dP( X_T^2 \geq T) = 0 =I_0(0),$$
which achieves the proof of Corollary \ref{C-LDP-UT}.
\demend

\vspace{-4ex}
\subsection*{}
\begin{center}
{\rm {\small B.3. THE EXPLOSIVE CASE.}}
\end{center}


\subsection*{Proof of Theorem \ref{T-LDP-EC}.}

The proof of Theorem \ref{T-LDP-EC} can be handled exactly as that of Theorem \ref{T-LDP-SC}
by taking $\theta > 0$.
However, in contrast with the stable case, the origin does no longer belong to 
the interior of $\cD_L$ and the function $L$, given by \eqref{DEFLS}, is not essentially smooth.
This is the reason why we are only allowed to apply the weakest version of the G\"artner-Ellis Theorem \ref{T-GE}.
\demend

\subsection*{Proof of Corollary \ref{C-LDP-EC}.}
We shall proceed as in the proof of Corollary \ref{C-LDP-UT} by considering rare events 
$\{\wh{\theta}_T \leq c\}$ and $\{\wh{\theta}_T \geq c\}$,
for $c<\theta$ and $c>\theta$, respectively.
First of all, we already saw that for any negative $c$,
$\dP(\wh{\theta}_T \leq c)= \dP(V_T \in \Delta_c)$
where $\Delta_c$ is the compact set of $\dR^2$ given, for $a_c(x)=(x^2 -1)/2c$, by 
$$\Delta_c = \Bigl\{ (x,y) \in \dR^2 \text{ such that } |x| \leq 1 \text{ and } y \in [0,a_c(x)]\Bigr\}.$$
Since $\Delta_c \cap \mathcal{F} \neq \emptyset$, it follows from Theorem \ref{T-LDP-UC} together with
Remark \ref{REM-TthetaC0} that 
$$\lim_{T \rightarrow + \infty} \frac{1}{T} \log \dP(\wh{\theta}_T \leq c) = -\inf_{(x,y) \in \Delta_c} \cI_{\theta}(x,y).$$
However, the rate function $\cI_\theta$ has no critical points on $\mathcal{F}$ and $\cI_\theta(x,0)=+\infty$.
Consequently,
$$\lim_{T \rightarrow + \infty} \frac{1}{T} \log \dP(\wh{\theta}_T \leq c) 
= -\inf_{|x| <1} \cI_\theta(x, a_c(x)) 
= -\cI_\theta\left(0,-\frac{1}{2c}\right).$$
In particular, as soon as $c <-\theta$,
$$\lim_{T \rightarrow + \infty} \frac{1}{T} \log \dP(\wh{\theta}_T \leq c) = \frac{(c-\theta)^2}{4c} = -I_{\theta}(c),$$
while, for $-\theta \leq c < 0$,
$$\lim_{T \rightarrow + \infty} \frac{1}{T} \log \dP(\wh{\theta}_T \leq c) = -\theta = -I_{\theta}(c).$$
From now on, assume that $0 \leq c <\theta$. 
We have for any $\alpha >1/2\theta$, 
\begin{equation}
\label{DECOUCE}
 \dP(\wh{\theta}_T \leq c) = \dP\Bigl(\wh{\theta}_T \leq c,\frac{S_T}{T} \leq \alpha\Bigr) +  
 \dP\Bigl(\wh{\theta}_T \leq c,\frac{S_T}{T} > \alpha\Bigr).
\end{equation}
On can remark that 
$$\dP\Bigl(\wh{\theta}_T \leq c,\frac{S_T}{T} \leq \alpha \Bigr) = \mathbb{P} (V_T \in \Delta_{c,\alpha})$$
where $\Delta_{c,\alpha}$ is the compact set of $\mathbb{R}^2$ defined by
$$\Delta_{c,\alpha}= \Bigl\{ (x,y) \in \mathbb{R}^2 \text{ such that } 0\leq y \leq \alpha \text{ and } y \geq a_c(x) \Bigr\}.$$
Since $\Delta_{c,\alpha} \cap \cF \neq \emptyset$, we obtain from Theorem \ref{T-LDP-UC} together with Remark \ref{REM-TthetaC0} that
$$\lim_{T \rightarrow + \infty} \frac{1}{T} \log \dP(V_T \in \Delta_{c,\alpha}) = -\inf_{(x,y) \in \Delta_{c,\alpha}} \cI_\theta(x,y).$$
After some straightforward calculations, we find that
\begin{equation}
\label{PRET1}
\lim_{T \rightarrow + \infty} \frac{1}{T} \log 
\dP\Bigl(\wh{\theta}_T \leq c,\frac{S_T}{T} \leq \alpha\Bigr) = 
-\cI_\theta \Bigl(0,\frac{1}{2\theta}\Bigr) = - \theta 
= -I_\theta(c).
\end{equation}
It now remains to show that the remainder term of \eqref{DECOUCE} is negligeable. 
It follows from Markov's inequality
that for any negative $\lambda$ and for any positive $\mu$,
\begin{eqnarray}
\dP\Bigl(\wh{\theta}_T \leq c,\frac{S_T}{T} > \alpha\Bigr) &=& \dP\Bigl( X_T^2 -2cS_T \leq T,
 S_T > \alpha T\Bigr), \notag \\
 & \leq & \exp\Bigl(-T(\lambda + \mu \alpha)\Bigr) \dE\Bigl[ \exp\Bigl( \lambda X_T^2 +(\mu-2\lambda c)S_T\Bigr) \Bigl], \notag \\
 & \leq & \exp\Bigl(-T\Bigl((\lambda + \mu \alpha)-\Lambda_T(\lambda, \mu-2\lambda c) \Bigr) \Bigr). 
\label{PRET2}
\end{eqnarray}
By setting $\lambda=(c-\theta)/2$ and $\mu=(c-\theta)^2/4$, one can check 
that the couple $(\lambda,\mu-2\lambda c)$ belongs to the effective domain
$\cD_\Lambda$ given in Lemma \ref{L-LCGFXcarre}. Hence, 
we obtain from \eqref{PRET2} that for $\alpha$ and $T$ large enough, 
\begin{equation}
\label{PRET3}
\dP\Bigl(\wh{\theta}_T \leq c,\frac{S_T}{T} > \alpha\Bigr) \leq \exp(-2\theta T).
\end{equation}
As a consequence, we deduce from \eqref{DECOUCE}, \eqref{PRET1} and \eqref{PRET3} that
for any $0 \leq c <\theta$,
$$\lim_{T \rightarrow + \infty} \frac{1}{T}\log \dP(\wh{\theta}_T \leq c) = -\theta = -I_\theta(c).$$
Finally, we shall investigate the case $ c >\theta$. 
We have for any $\alpha >0$, 
\begin{equation}
\label{DECOUCF}
 \dP(\wh{\theta}_T \geq c) = \dP\Bigl(\wh{\theta}_T \geq c,\frac{|X_T|}{\sqrt{T}} \leq \alpha \Bigr) +  
 \dP\Bigl(\wh{\theta}_T \geq c,\frac{|X_T|}{\sqrt{T}} > \alpha \Bigr).
\end{equation}
As in the proof of Corollary \ref{C-LDP-UT},
$$\dP\Bigl(\wh{\theta}_T \geq c,\frac{|X_T|}{\sqrt{T}} \leq \alpha \Bigr) = \dP (V_T \in \Delta_{c,\alpha})$$
where $\Delta_{c,\alpha}$ is the compact set of $\mathbb{R}^2$ defined by
$$\Delta_{c,\alpha}= \Bigl\{ (x,y) \in \dR^2 \text{ such that } 1 \leq |x| \leq \alpha \text{ and } y \in [0,a_c(x)]\Bigr\}.$$
Since $\Delta_{c,\alpha} \cap \cF \neq \emptyset$, it follows from Theorem \ref{T-LDP-UC} together with Remark \ref{REM-TthetaC0} that
$$\lim_{T \rightarrow + \infty} \frac{1}{T} \log \dP(V_T \in \Delta_{c,\alpha}) = -\inf_{(x,y) \in \Delta_{c,\alpha}} \cI_\theta(x,y).$$
Furthermore, denote
$$
\alpha_c(\theta)=\sqrt{\frac{c+\theta}{c-\theta}}.
$$
After some straightforward calculations, we obtain that, as soon as $\alpha\ge \alpha_c(\theta)$,
\begin{equation}
\label{PREF1}
\lim_{T \rightarrow + \infty} \frac{1}{T} \log \dP\Bigl(\wh{\theta}_T \geq c,\frac{|X_T|}{\sqrt{T}} \leq \alpha \Bigr) = 
-\cI_\theta\Bigl(\alpha_c(\theta),\frac{1}{c-\theta}\Bigr) = \theta-2c = -I_\theta(c).
\end{equation}
Using once again Markov's inequality, we have for any positive $\lambda$ and $\mu$,
\begin{eqnarray}
\dP\Bigl(\wh{\theta}_T \geq c,\frac{|X_T|}{\sqrt{T}} > \alpha \Bigr) &=& \dP\Bigl( X_T^2 -2cS_T \geq T,
 X_T^2 > \alpha^2 T\Bigr), \notag \\
 & \leq & \exp\Bigl(-T(\lambda + \mu \alpha^2)\Bigr) \dE\Bigl[ \exp\Bigl( (\lambda+\mu)X_T^2 -2\lambda cS_T\Bigr) \Bigl], \notag \\
 & \leq & \exp\Bigl(-T\Bigl((\lambda + \mu \alpha^2)-\Lambda_T(\lambda+\mu, -2 \lambda c) \Bigr) \Bigr). 
\label{PREF2}
\end{eqnarray}
By choosing $\lambda=(c^2-\theta^2)/4c$ and $\mu = (c-\theta)^2/8c$, it is not hard to see that the couple 
$(\lambda+\mu,-2c\lambda )$ belongs to the effective domain
$\cD_\Lambda$ given in Lemma \ref{L-LCGFXcarre}. Hence, we obtain from \eqref{PREF2} that for $\alpha$ and $T$ large enough, 
\begin{equation}
\label{PREF3}
 \dP\Bigl(\wh{\theta}_T \geq c,\frac{|X_T|}{\sqrt{T}} > \alpha \Bigr) \leq \exp(-2(2c-\theta)T).
\end{equation}
Therefore, it follows from the conjunction of \eqref{DECOUCF}, \eqref{PREF1} and \eqref{PREF3} that
for any positive $c >\theta$,
$$\lim_{T \rightarrow + \infty} \frac{1}{T}\log \dP(\wh{\theta}_T \geq c) = \theta-2c = -I_\theta(c),$$
which completes the proof of Corollary \ref{C-LDP-EC}.
\demend

\section*{Appendix C \\ Proofs of CI results.}
\renewcommand{\thesection}{\Alph{section}}
\renewcommand{\theequation}{\thesection.\arabic{equation}}
\setcounter{section}{3}
\setcounter{equation}{0}
\label{PRCI}

It follows from \eqref{OUPS} and \eqref{DEFTHETAHAT} that
\begin{equation*} 
\wh{\theta}_T -\theta= \frac{M_T}{S_T}
\hspace{1cm}\text{where}\hspace{1cm}
M_T=\int_0^T X_t dB_t.
\end{equation*}
The sequence $(M_T)$ is a locally square-integrable martingale with $M_0=0$.
For all $a \in \dR$, denote
$$
W_T(a)=\exp\Bigl(a M_T - \frac{a^2}{2} S_T\Bigr).
$$
It is well-known that $(W_T(a))$ is a positive supermartingale
such that $\dE[W_T(a)] \leq 1$, see e.g. \cite{BJY86}. We are now in position
to prove Theorem \ref{T-CI}.

\subsection*{Proof of Theorem \ref{T-CI}.} Via the same lines as in the proof of Theorem 3.25 in \cite{BDR15},
see also the Appendix of \cite{DP99}, we claim that for any positive $x$,
\begin{equation} 
\label{CIP1}
\dP(|\wh{\theta}_T -\theta| \geq x) \leq 2 \inf_{p>1}\left(\dE\Bigl[\exp\Bigl(-(p-1)\frac{x^2}{2}
S_T\Bigr)\Bigr]\right)^{1/p}\!\!\!\!.
\end{equation}
As a matter of fact, we have $\dP(|\wh{\theta}_T -\theta| \geq x)= 2 \,\dP(A_T)$ where
$A_T=\{M_T\geq x S_T\}$.
We deduce from Markov's inequality together with Holder's inequality that,
for any positive $a$ and for all $q>1$,
\begin{eqnarray}
\label{CIP2}
\dP(A_T)&\leq &\dE\Bigl[\exp\Bigl(\frac{a}{q}M_T - \frac{ax}{q}S_T\Bigr)
\rI_{A_T}\Bigr],\nonumber \\
&\leq &\dE\Bigl[(W_T(a))^{1/q}
\exp\Bigl(\frac{a}{2q}(a- 2x)S_T\Bigr)\rI_{A_T}\Bigr],\nonumber \\
&\leq &
\left(\dE\Bigl[\exp\Bigl(\frac{ap}{2q}(a- 2x)S_T\Bigr)\Bigr]\right)^{1/p}
\end{eqnarray}
where $p$ and $q$ are H\"older conjugate exponents, since $\dE[W_T(a)] \leq 1$. Consequently, 
we deduce from \eqref{CIP2} with $a=x$ and the elementary fact that $p/q=p-1$, that
\begin{equation*}
\dP(A_T)\leq \inf_{p>1}\left(
\dE\Bigl[\exp\Bigl(-(p-1)\frac{x^2}{2}S_T\Bigr)\Bigr]\right)^{1/p}
\end{equation*}
which immediately leads to \eqref{CIP1}. It only remains to find a suitable upper bound
for the Laplace transform of $S_T$. Let $b=-(p-1)x^2/2$ where $p>1$.
In the stable and unstable cases $\theta \leq 0$, it follows from \eqref{VARS} and
\eqref{DECMGFLS2} with $a=0$ and $\varphi = \sqrt{\theta^2-2b}$, that
\begin{eqnarray} 
\dE[\exp(bS_T)] & = &\exp\Bigl(\frac{T}{2}\bigl(\vp-\theta\bigr) -\frac{1}{2}\log \bigl(1+ (\vp-\theta)\sigma_T^2\bigr)\Bigr), \notag \\
& \leq & \exp\Bigl(-\frac{T}{2}\bigl(\vp+\theta\bigr) -\frac{1}{2}\log \Bigl(\frac{\vp-\theta}{2\vp}\Bigr)\Bigr)
\label{CIP3}
\end{eqnarray}
as $\vp>-\theta$. Consequently, we obtain from \eqref{CIP1} and \eqref{CIP3}
that, for any positive $x$,
\begin{equation} 
\label{CIP4}
\dP(|\wh{\theta}_T -\theta| \geq x) \leq 2 \inf_{p>1}
\exp\Bigl(-\frac{1}{2p}\Bigl( T\bigl(\vp+\theta\bigr) +\log \Bigl(\frac{\vp-\theta}{2\vp}\Bigr)\Bigr)\Bigr).
\end{equation}
On the one hand, if $\theta=0$ and $y=\sqrt{(p-1)x^2}$, \eqref{CIP4} immediately implies that
\begin{equation} 
\label{CIP5}
\dP(|\wh{\theta}_T | \geq x) \leq 2 \inf_{y>0}
\exp\Bigl(-\frac{x^2}{2}\Bigl( \frac{Ty -\log 2}{x^2+y^2}\Bigr)\Bigr).
\end{equation}
On the other hand, if $\theta<0$ and $y=-(\vp + \theta)/\theta$, it follows from \eqref{CIP4} that
\begin{equation} 
\label{CIP6}
\dP(|\wh{\theta}_T -\theta| \geq x) \leq 2 \inf_{y>0}
\exp\Bigl(-\frac{x^2}{2}\Bigl( \frac{-\theta T y + \log(y+2) -\log(2(y+1))}{x^2+\theta^2y(y+2)}\Bigr)\Bigr).
\end{equation}
By the same taken, in the explosive case $\theta>0$, we find that
\begin{equation} 
\dE[\exp(bS_T)] \leq  \exp\Bigl(\frac{T}{2}\bigl(\vp-\theta\bigr) -\frac{1}{2}\log \Bigl(\frac{\vp+\theta}{2\vp}\Bigr)\Bigr)
\label{CIP7}
\end{equation}
where $\varphi = -\sqrt{\theta^2-2b}$. Therefore, if $y=-(\vp + \theta)/\theta$, we deduce from
\eqref{CIP1} and \eqref{CIP7} that, for any positive $x$,
\begin{equation} 
\label{CIP8}
\dP(|\wh{\theta}_T -\theta| \geq x) \leq 2 \inf_{y>0}
\exp\Bigl(-\frac{x^2}{2}\Bigl( \frac{\theta T(y+2) + \log y -\log(2(y+1))}{x^2+\theta^2y(y+2)}\Bigr)\Bigr).
\end{equation}
Finally, \eqref{CIGEN} follows from the conjunction of \eqref{CIP5}, \eqref{CIP6} and \eqref{CIP8},
which achieves the proof of Theorem \ref{T-CI}.
\demend

\vspace{-2ex}
\subsection*{Proof of Corollary \ref{C-CI-SC}.}
In the stable case $\theta<0$, we clearly have for any $y>0$, $h_T(y)>\ell_T(y)$ where
$$
\ell_T(y)=\frac{-\theta T y -\log 2}{x^2 +\theta^2y(y+2)}.
$$
The function $\ell_T$ reaches its maximum at
the value
$$
y_x=-\frac{1}{\theta T} \bigl(\log 2 + \sqrt{T^2x^2 -2\theta T \log 2+(\log 2)^2}\bigr).
$$
Putting this value into \eqref{CIGEN} immediately leads to \eqref{CI-SC}.
\demend

\vspace{-2ex}
\subsection*{Proof of Corollary \ref{C-CI-UC}.}
The proof of Corollary \ref{C-CI-UC} is left to the reader as it is exactly the same as that of
Corollary \ref{C-CI-SC}.
\demend

\vspace{-2ex}
\subsection*{Proof of Corollary \ref{C-CI-EC}.}
In the explosive case $\theta>0$, putting the value 
$$
y= \frac{\log 2}{\theta T}
$$
into \eqref{CIGEN} immediately leads to \eqref{CI-EC-xpetit}. 
\demend

\vspace{-2ex}
\bibliographystyle{acm}

\end{document}